%% file: filling3.tex
\documentclass[11pt]{amsart}
\usepackage{amsmath}
\usepackage{amssymb}
\usepackage{amsfonts}
\usepackage{amsthm}
\usepackage{verbatim}
\usepackage{amscd}
\usepackage{cite}
\usepackage{leftidx}
\usepackage{enumerate}
\usepackage{txfonts}
\usepackage{manfnt}
\usepackage{amscd}
\usepackage{mathpazo}
\usepackage[mathscr]{eucal}
\usepackage{xr}
\usepackage{hyperref}
\input{myheader2.tex}
\title { Unobstructedness of filling secants\\ and the Gruson-Peskine general projection theorem }
\normalsize
\author
{Ziv Ran}

\renewcommand{\theequation}
{\arabic{equation}}

\begin{document}

\pagestyle{plain}

\date {\today}
\address{\tiny  {\newline Ziv Ran \newline University of California
Mathematics Department\newline Big Springs  Surge Facility
\newline Riverside CA 92521
\newline ziv.ran @ ucr.edu}}
\thanks{arxiv.org/1302.0824}

\begin{abstract}
We prove an unobstructedness result for deformations of subvarieties constrained
by intersections with another,  fixed  subvariety. We deduce smoothness and expected-dimension results
for multiple-point loci of generic projections, mainly from a point or a line, or for fibres of embedding dimension
2 or less. 
\end{abstract}
\maketitle
The study of linear projections of a smooth projective variety $X\subset\P^N$, and the closely related study of multisecant
spaces of $X$, have occupied projective geometers for generations (see e.g. \cite{zak}).
Though many pathologies are possible for multisecant spaces,
projections from a generic center $\Lambda$, or equivalently the multisecant spaces containing $\Lambda$,
seem not to be afflicted by them, at least when $\Lambda$ is a point. This 
has led to the formulation of a folklore 'generic projection conjecture'
(about which we first heard from R. Lazarsfeld \emph{ca.} 1990): the projection
of $X$ from a generic point $p\in\P^N$ has only the 'expected' singularities (see below).
This statement is equivalent to one about families of secant lines to $X$ satisfying contact conditions,
filling up the ambient $\P^N$.  After numerous 
partial results including \cite{mather}, \cite{alzati-ottaviani},\cite{(n+2)sec}, \cite{beis}, 
the conjecture was recently proven by Gruson and Peskine \cite{grp} (which the reader
may consult  for further
introductory comments,   references, as well as examples and applications).\par
In this paper we will prove 3 results, each extending the Gruson-Peskine
Generic Projection Theorem: \par
-  Theorem \ref{point-projection-thm} is a version for arbitrary 
ambient spaces in place of $\P^N$,  in the form of a general result about  deformations of
rational curves on varieties constrained by contact conditions with a fixed subvariety: it is
shown that these deformations are well-behaved 
(unobstructed, of the expected dimension) provided they fill up the ambient space . 
\par - Theorem \ref{line-projection-thm}  
is a generic projection theorem in $\P^N$  (though without contact conditions)
for projections from arbitrary-dimension centers and their
fibres which have local embedding dimension 2, which includes all fibres
in the case of
projections from a generic \emph{line}. 
\par- Theorem \ref{curvilinear-thm} is a full generic projection theorem in $\P^N$, 
contact conditions included, for the case of curvilinear fibres (and arbitrary-dimension centers).\par
The difference between our approach and that of Gruson and Peskine  is, in a word, cohomology (albeit, of the most elementary kind). We will develop some deformation theory for secants to $X$,
then apply the 'filling' hypothesis to show that obstructions-
even obstruction \emph{groups}-  vanish.
Thus in outline, the proof of each of these theorems follows the same overall patern:\par
- the secant or contact conditions are  analyzed locally and encoded in certain secant or contact
\emph{sheaves} $M$ on the secant plane or curve $L$, which control the corresponding deformations;\par
- the filling hypothesis implies a certain generic spannedness
property for $M$;\par
- $L$ is isomorphic to $\P^\l$, often even $\P^1$, the generic spannedness implies
spannedness and the vanishing of higher cohomology, whence well-behaved deformations.\par
This note grew out of an attempt to understand the exciting result of Gruson and Peskine \cite{grp},
  viewed as a statement about multisecant lines which are mobile enough to go through
  a generic point of the ambient space.
  Though our proof is independent of 
  \cite{grp} and the connection to it may be hidden, it is nonetheless fundamental.
  It occurs above all through  the following heuristic idea gleaned from \cite{grp}:
  \ss
  {\textbf{Uniformity principle:}}{\textit{\ The family of secants to $X$ behaves in the same way at all points off $X$}} \ss 
  For example: if $L$ is postulated to contain a $k$-tuple $Z\subset X$ and  $p\notin Z$, we may consider whether under infinitesimal motion of $L$ as $k$-secant to $X$, and $p$ with it,
  $p$ fills up (the tangent space of) the ambient projective space. The uniformity principle says that the answer
  is \emph{independent} of the point $p$. This is a shockingly powerful conclusion: it implies 
  inter alia that if $k$-secants to $X$ fill up the ambient space, they cannot all meet another variety of codimension $>1$,
  and cannot meet $X$ itself in $>k$ points. This is quite close to saying that the filling $k$-secants of $X$ 
  are well-behaved.\par
  A (future ?) Google-type search on 'uniformity principle' in Algebraic Geometry is likely to bring up 
  prominently the fact  that on a $\P^1$,
  a generically spanned vector  bundle is spanned everywhere (parenthetically, such bundles also have no higher cohomology).
Our idea is that the two uniformity principles are related, the first
being ultimately a consequence of the second.
 Applying this idea to the bundles and sheaves occurring
 in the deformation theory of secants is  the basis for our 
 cohomological/deformation-theoretic approach outlined above.\par
The paper is organized as follows.In \S 1 we give some precise definitions of secants and their scheme structure, as well as the notion of 'filling' and its infinitesimal analogue, which requires a little care. In \S 2 we study finite schemes, especially curvilinear ones, and their deformation spaces.In \S 3 we
study various sheaves related to deformations
of secants, which are fundamental for our deformation-theoretic approach. Then in \S 4,5,6 we prove the aforementioned  3 theorems, along with a few corollaries. \par
 We thank Yeongrak Kim for helpful discussions and the referees for 
 constructive comments and suggestions and stimulating questions that
 have greatly improved the paper. 
 Hopefully, our indebtedness to the breakthrough work of Gruson
  and Peskine would be obvious to anyone.\par
\emph{Conventions.} In this paper we work over $\C$; sometimes for added generality, we will work
in the complex- analytic category (the kind of analytic varieties we have in mind are open subsets
of algebraic varieties). If $X$ is an algebraic  scheme
or analytic space, and $a$ is any negative number, the statements '$X$ has dimension $a$' and '$X$ is empty' are by definition equivalent. Similarly, if $X$ has dimension $n\geq 0$ and $m>n$ and $Y\subset X$, the statements '$Y$ is empty' and '$Y$ has codimension $m$
in $X$' are equivalent. We use $\l$ to denote the length of a module and $\l_z$ its local length at a point $z$. 
A statement about a 'general'
point $y$ of an analytic variety $Y$ is by definition true if it holds for all $y$ in the complement
of a nowhere dense analytic subset of $Y$.  \par
See \cite{sernesi}, \cite{lehn-montreal}, \cite{lichtenbaum}
 for some foundational results on Hilbert schemes
and deformation theory. A more general setting based on the language of Lie theory is presented in
\cite{atom} and \cite{sela}

\section{Secants and fillers}\label{secants-and-fillers-sec}
\subsection{Secants }\label{secants}
In this subsection, we will define secant loci set-theoretically.
Let $P$ be a nonsingular, quasi-projective
 variety and
 $\gg$ a connected smooth open subset of a suitable Hilbert scheme
 of $P$ parametrizing smooth connected complete unobstructed
 subvarieties
  $L\subset P$, to be called \emph{flats}.
  In applications, 
 flats $L$ will be mostly isomorphic to $\P^m$ (often $\P^1$).
 If in fact $L\simeq \P^1$, it is well known that
 \eqspl{}{
 \dim(\gg)=-K_P.L+\dim(P)-3.
 }
In any case, $\gg$ comes equipped with a tautological family $\mathcal L\to\gg$.
 
Let $X$ be a nonsingular, locally closed subvariety of $P$, closed in a neighborhood of $L$. 
Throughout, $P$ and $X$ will be considered fixed, though $L$ will move.
 Let $X\sbr k.$ and
  $X\spr k.$ denote, respectively,  the Hilbert scheme of length-$k$ subschemes of $X$, 
  and the $k$-th symmetric product. There is a well-known cycle (or 'Hilb to Sym')
  morphism (see \cite{lehn-montreal} \S 3.2 or \cite{structure}, \S 1.2) 
  \[\mathfrak c:X\sbr k.\to X\spr k.\ \  .\]
  Then $X\sbr  k.$
  is canonically stratified by  
 closed subschemes $X\sbr {k.}.$ for all partitions $(k.)$ of weight $k$, 
  where $X\sbr {k.}.$ consists of the finite schemes $z$ such that $\mathfrak c(z)=\sum k_ip_i$, 
  $p_i\in X$ not necessarily distinct. Such $z$ is said to be of cycle type $(k.)$, and properly so
 when the $p_i$ are distinct. We have that $X\sbr{k.}.\subset X\sbr{m.}.$ whenever the partition $(m.)$
  is a refinement of $(k.)$ , i.e. obtained by subdividing some blocks.\par

For a partition $(k.)= (k_1,...,k_r)$,  $L$
is said to be \emph{$(k.)$-secant} to $X$ if the schematic intersection $L\cap X$ is of finite length and
contains a subscheme  of type 
$(k.)$ , and
 \emph{proper} as such if  $L\cap X$ itself is properly of type $(k.)$. $L$ is said to be $k$-secant to $X$
 if the length of $L\cap X$ is finite and at least $k$ (equivalently,  $L$ is
 $(1^k)$-secant, i.e. properly $(l.)$-secant for some partition $(l.)$ of weight
 $\sum l_i=l\geq k$).\par
 Let \[\tilde S_{(k.)}\subset \gg\times X\sbr {k.}.\]
 denote the locus of pairs $(L, z)$ where $z\subset L\cap X$, and let $S_{(k.)}$ denote its projection
 to $\gg$. Thus,
$S_{(k.)}\subset\mathbb G$ is the locus of flats that are $(k.)$-secant to $X$. 
Note that $\tilde S_{(k.)}, S_{(k.)}$ are closed subsets  of $\gg\times X\sbr{k.}., \gg $ respectively
and the map $\tilde S_{(k.)}\to S_{(k.)}$ is finite-to-one and
bijective over the open subset parametrizing $L$s such that $L\cap X$ has length $k$. 
Moreover if either $L$ is a smooth curve or $k_i=1$ for all $i$,
 or more generally if the intersections $X\cap L$ are (locally) curvilinear schemes,
 then $X\sbr{k.}.$ has a well-defined structure
 of smooth closed subscheme of the smooth curvilinear Hilbert scheme
 $P\sbr k._0$ (see \S \ref{curvilin-sec}) in a neighborhood of $z$, hence $\tilde S_{(k.)}$ 
 is endowed with a scheme structure as the pullback of $X\sbr{k.}.$
 by the natural map $\mathcal L\sbr k._\gg\to P\sbr k.$ where $\mathcal L\sbr k._\gg$
 is the relative Hilbert scheme.
Then $S_{(k.)}$ is endowed with the image scheme structure.
Note that $S_{(k.)}$ is  closed in $\gg$ if $X\subset P$ is.
For $L\in S_{(k.)}$, the intersection $L\cap X$may have length greater than
the weight of $(k.)$ and possibly $\infty$.\par
 If $Y\subset P$ is an analytic submanifold of a tubular (analytic) neighborhood
 of a fibre $L$, or germ of one, we denote by 
 $\tilde S_{(k.),Y}$ and $S_{(k.),Y}$ the appropriate loci (or schemes) corresponding
 to deformations of $L$ within $Y$.\par
 One can similarly  define a secant scheme $\tilde S_k$,
   postulating just the total length of $L\cap X$,
  as the pullback of closed subscheme $X\sbr k.\subset P\sbr k.$
  by the natural map $\L\sbr k._\gg\to P\sbr k.$.
\subsection{Fillers} \label{fillers}
The working method of this paper is to study filling families
by infinitesimal methods, viz. the powerful tools of Grothendieck's deformation
theory. There is a certain subtlety involved here, because
  the filling notion is essentially of global
  character and meaningful mainly for (quasi-projective) \emph{varieties}, 
  while deformation theory applies directly to formal completions, hence to \emph{germs}.
   That some care is required is illustrated by the fact that
    a map of point germs can be birational without being smooth
    at the unique (closed) point.\par
  We are thus led to
  the notion
  of 'infinitesimal filling', which serves to bridge between filling notions and infinitesimal methods.
 
 \begin{defn}\label{filling} (i) Let $Y$ be a smooth analytic open subset
 of an irreducible projective
  variety $\bar Y$ and $\L/B$ a smooth flat family
 of closed subvarieties of $\bar Y$ contained in $Y$,
  parametrized by an irreducible analytic variety $B$, such that $\L$
 is irreducible. The family is said to be \emph{filling}
 if the natural projection $\pi_Y:\L\to Y$ is dominant,
 i.e. its image contains an analytic open set.\par
(ii) Let $Y$ be a smooth local germ along a smooth 
irreducible subvariety $L$ of a smooth quasi-projective variety $P$. 
 A local deformation $\L\to\Spec(A)$ of $L$ within $Y$ is said to be an
 \emph{ infinitesimally filling family} for $Y$
 if there exists an irreducible closed subscheme $S$ of $\Spec(A)$, such that the  map
 $(\L_S)_{\red}\to Y$ induces a surjection on (Zariski) tangent spaces $T_p(\L_S)_{\red}\to T_pY$
 for a general point $p\in L$.
 \end{defn}
 \begin{rem}
In Part (i) above, $Y$ could equal $\bar Y$ or be Zariski open it it. 
A typical non-algebraic $Y$ can be a tubular neighborhood of some fibre of $\L/B$.
 \end{rem}
 \begin{rem}
 To be clear, we are talking here about an \emph{arbitrary} member of the $\L/B$ family
 through a \emph{general} point of $Y$, which is tantamount to talking about
 an \emph{arbitrary} point of a \emph{general} fibre of a suitable morphism; we aim to invoke generic
 smoothness, which says that the morphism is smooth at such a point.
 \end{rem}
A link between the two types of 'filling' is the following.
  \begin{lem}\label{filling-lem}
  In the situation of Definition \ref{filling}, (i),
  assume additionally that $\L\to B$ is proper.
  Then if $p\in Y$ is general and $L\to b$ is an \emph{arbitrary} fibre of 
    the filling family $\L/B$ going through $p$,
    then the germ of $\L/B$ along $L\to b$, i.e. $\L\times_B\Spec(\O_{B,b})$, 
    is infinitesimally filling, hence so is any larger family.
  \end{lem}
  \begin{proof}
  $\L\to B$ extends to
  a proper flat algebraic family $\bar \L\to\bar B$. By enlarging $B$, we may assume $B$ coincides
  with the open subset of $\bar B$ corresponding to fibres contained in $Y$.
  Let $\hat\L\to\bar \L$ be a desingularization. Then the  fibre $F$ of 
  $\bar\L\to\bar Y$ over a general point $p\in Y\subset\bar Y$ is
  smooth, i.e. for \emph{all} $f\in F$, the derivative map $T_f\hat \L \to T_pY$ is surjective.
  In particular, if $f$ is in the open subset of $F$ going into 
  a point of $\L$, say $q\in (L\to b)$, then the latter derivative factors
  trough $T_q\L\to T_pY$, so the latter is surjective as well. We may identify $q$ with $p$,
  so the infinitesimal filling property follows.
  
  \end{proof}
  \begin{rem}
  In applications, $\L\to B$ will be a smooth morphism but $B$, hence $\L$ as total
  space, may have arbitrary singularities.
  \end{rem}

   There are analogues of the notions of filling and infinitesimally filling,
   and of their relation, as in the last Lemma, \emph{with respect to $\Lambda$},
    where 'point' is replaced by $\Lambda$, a subvariety
moving in a smooth filling family on $Y$: the precise definition
is the following
\begin{defn} \label{inf-filling-higher}
Notation as in Definition \ref{filling}, let $\L_i\to B_i$ be proper and filling families, $i=1,2$.
Let $B\subset B_1\times B_2$ be a subset parametrizing pairs $(\Lambda, L)$
with $\Lambda\subset L$. Then the family of pairs parametrized by $B$ is said to be
filling with respect to $\Lambda$ if the projection $B\to B_1$ is dominant. The family
is said to be infinitesimally filling with respect to $\Lambda$ at a given $L$ corresponding
to $b_2\in B_2$ is for a general $b\in B$ projecting to $b_2$,
the projection of Zariski tangent spaces $T_bB\to T_{b_1}B_1$ is surjective.
\end{defn}
The following extension of Lemma \ref{filling-lem} is proved in the same way:
\begin{lem}\label{filling-higher-lem}
Notations as above, if $b\in B_1$ is general and $b\in B$ projecting to $b_1$ is arbitrary,
then the map of Zariski tangent spaces $T_bB\to T_{b_1}B_1$ is surjective.
\end{lem}

   \section{Curvilinear schemes}\label{curvilin-sec}
  \subsection {Basics}
  Let $L$ be a nonsingular variety. Let $L\sbr k.$,  $L\sbr k._0\subset L\sbr k.$, and
  $L\spr k.$ denote, respectively,  the Hilbert scheme of length-$k$ subschemes of $L$, 
  its open subset consisting of \emph{curvilinear} schemes, i.e. those of local embedding
  dimension 1 or less (see below), 
  and the $k$-th symmetric product. There is a well-known cycle (or 'Hilb to Sym')
  morphism \[\mathfrak c:L\sbr k.\to L\spr k.\ \  .\]
  Then $L\sbr  k.$
  is canonically stratified by  closed
  subsets $L\sbr {k.}.$ for all partitions $(k.)$ of weight $k$, 
  where $L\sbr {k.}.$ consists of the finite schemes $z$ such that $c(z)=\sum k_ip_i$, 
  $p_i\in X$ not necessarily distinct. Such $z$ is said to be of cycle type $(k.)$, and properly so
  if the $p_i$ are distinct. 
  The open subset $L\sbr{k.}._0\subset L\sbr{k.}.$ is well-understood and has a canonical scheme 
  structure..
  We have that $L\sbr{k.}._0\subset L\sbr{m.}._0$ whenever the partition $(m.)$
  is a refinement of $(k.)$ , i.e. obtained by subdividing some blocks.
  The set of curvilinear schemes properly of type $(k.)$ is well known (compare 
 \S \ref{emb-def-sec} below) to be
  a smooth open subset 
  $L\sbr {k.}._0\subset L\sbr k.$ of dimension $k(\dim(L)-1)+r$ where $(k.)=(k_1\geq...\geq k_r>0)$.
\subsection{Embedded deformations}\label{emb-def-sec}
A local curvilinear scheme $z$ of length $k\geq 2$ is 
isomorphic to $\Spec(\C[x]/(x^k))$. It has unobstructed
versal deformation space 
\[\Def(z)\simeq T^1_z\simeq\O_z/(x^{k-1})\]
 with basis $1, x, ...,x^{k-2}$. 
 Identifying $\Def(z)$ with the set of polynomials \[\{x^k+b_{k-2}x^{k-2}+...+b_0:b_i\in \C\},\] it is stratified by loci
 $D_{(k.)}$ corresponding to partitions of weight $k$. If the partition
 $(k.)$ is written as $(l_i^{e_i})$ with $(l.)$ strictly decreasing,
  then $D_{(k.)}$ can be identified with the following collection of factored polynomials:
  \[D_{(k.)}=\{\prod\limits_i\prod\limits_{j=1}^{e_i}(x-a_{i,j})^{l_i}: a_{i,j}\in\C, \sum\limits_{i,j} a_{i,j}=0\}.\]
  Identifying, for fixed $i$,  $\prod\limits_j(x-a_{i,j})$ with a point in $\sym^{e_i}(\A^1)$ and, in turn, with
  a point in $\A^{e_i}$ via the elementary symmetric functions $\sigma_r(a_{i,1},...,a_{i,e_i})$,
$D_{(k.)}$ is bijectively and birationally (but not isomorphically)
  parametrized by the subspace of $\prod \sym^{e_i}(\A^1)$ defined by 
  $\sum\limits_i\sigma_1(a_{i,1},...,a_{i,e_i})=0$.
  
 \par
When $z$ is embedded in a smooth variety,
its ideal $\I_z$ has the form
\[ (f_1=x_1^k, f_2=x_2,...,f_n=x_n)\]
for a suitable regular system of parameters $x_1,...,x_n$; first-order deformations of $\I_z$ take the form
\[(f_1+b_1,..., f_n+b_n), b_i=\sum\limits_{j=0}^{k-1} b_{i,j}x_1^j, i=1,....,n.\]
The obstructions vanish, because $\I_z$ is a local complete intersection. The locus of deformations
preserving the cycle type, or equivalently, locally trivial deformations, 
is the (smooth) subscheme given by the condition
\[b_{1,j}=0, j<k-1.\]
For a general, non-local curvilinar scheme $Z=\coprod z_i$, the versal deformation and Hilbert scheme are
the product of those for the $z_i$.
\subsection{Abstract deformations}

Up to a smooth factor, deformations of $Z$ on a smooth variety
$L$ are the same as deformations of $z$ on a smooth
curve-germ $C$ containing $Z$ as given in the above notation by $x_2=...=x_n=0$. 
In fact, if $Z$ is properly of type $(k.)$, i.e.  has multiplicity exactly $k_i$ at $p_i$,
and $z_i$ denotes the part of $Z$ supported at $p_i$, 
then the versal deformation  $\Def(Z)$ of $Z$ as abstract scheme splits as 
$\prod \Def(z_i)$ . There is a local classifying map from the Hilbert scheme of $L$:
\eqspl{}{
\rho:L\sbr {k}.\to \Def(Z).
} On the level of 1st-order deformations, this corresponds to the natural surjective map
\[N_{Z,L}\to T^1_Z\] 
where $N_{Z,L}=\Hom(\I_{Z, L}, \O_Z)$ is the lci normal bundle.
Denoting the kernel of this map by $N'_{Z,L}$, the 'locally trivial' normal bundle,
and identifying $T_L\otimes\O_Z\simeq\mathrm{Der}(\O_L, \O_Z)$,
we have exact sequences
\eqspl{}{&\exseq{T_{L/Z}}{T_L\otimes\O_Z}{N'_{Z,L}},\\ 
&\exseq{N'_{Z,L}}{N_{Z,L}}{T^1_Z}}
  where $T_{L/Z}$ is the derivations presrving $I_Z$.
 In particular, the morphism $\rho$ is smooth.

Consequently, if $Z$ is properly of type $(k.)$ then $L\sbr{k.}.$ and $L\sbr k., k=\sum k_i$ are smooth
at $Z$. Moreover, if $L\sbr{m.}.\subset L\sbr k.$ is any other stratum, necessarily
containing $L\sbr k.$,  containing $Z$, then the singularity
of $L\sbr {m.}.$ at $Z$ is the same up to a smooth factor
as that of $C\sbr {m.}.$ and the normalization of $L\sbr{m.}.$
is smooth. Indeed the normalization in question is just $\prod (C\spr {b_i}.)$
where the $b_i$ are the multiplicities of distinct-size blocks of $(m.)$, i.e.
 $(m.)=(a_1=...=a_1>a_2=...)$ with each $a_i$ occurring $b_i$ times.  
For instance, if $(k.)=(3), (m.)=(2,1)$ then
$C\sbr {m.}.$, hence $L\sbr {m.}.$ is locally $(\mathrm{smooth}\times{\mathrm{ordinary\ cusp}})$ .
The identification of $\prod (C\spr {b_i}.)$ with the normalization of $C\sbr {m.}.$ is immediate from
the fact that $\prod (C\spr {b_i}.)$ is smooth and the map $\prod (C\spr {b_i}.)\to C\sbr {m.}.$ is 
finite  and birational.
 \par 
  
  \par
  In the situation of Theorem \ref{point-projection-thm}, note that 
   there is an induced map near $(L,Z)$  
   \eqspl{mu}{\mu:\tilde S_{k,Y}\to \Def(Z)}
  such that $\tilde S_{(k.), Y}=\mu\inv(D_{(k.)})$. 
  The theorem's assertions are mostly covered by the statement
  that $\mu$ is a smooth morphism, which we will prove in \S \ref{proof}. The remaining
  assertions have to do with the case of intersections
  that are 'excessive' , i.e. have length $>k$. To prepare  this we next turn to
  nested pairs of curvilinear schemes.
  \par
  In \cite{sela} one can find a description of the 'tangent complex' of any affine scheme
  like $Z$ as a \dgla, whose deformation theory as such coincides with that of $Z$;
  similarly, when $Z$ is embedded in a smooth scheme $L$, the normal complex
  (or sheaf, when $Z$ is a local complete intersection) 
   $N_{Z/L}$ has the structure of 'Lie atom' whose deformation theory is that of
   embedded deformations of $Z$ in $L$.
  
  \subsection{Pairs}\label{pairs}
  We consider nested pairs of curvilinear schemes (cf. \cite{sernesi}, \S 4.5 for general flag Hilbert schemes).
  \begin{lem}\label{pairs-lem}
  Let $C$ be a smooth curve and denote by $C\spr d,k.\subset C\spr d.\times C\spr k.$
  denote the locus of nested pairs of schemes $z_d\leq z_k$ and by $C^{(d.),k}\subset C\spr d,k.$
  the sublocus where $z_d$ is of type $(d.)$ where$(d.)$ is a partition of weight $d$.
  Then the projection $C^{(d.), k}\to C^{(d.)}$ is a smooth morphism., hence the singularities
  of $C^{(d.), k}$ are the same up to a smoth factor as those of $C^{(d.)}$.
  \end{lem}\begin{proof}
  The projection in question is a base-change of the projection $C^{d,k}\to C^d$, so it suffices
  to prove the latter morphism is smooth. For this we may work locally on $C$ and assume $C=\A^1, z_k=(x^k)$ (hence $z_d=(x^d)$).
  Denote by $V_{d}$ the space of monic polynomials of degree $<d$ in $x$,
  identified with the local Hilbert scheme of $(x^d)$ on $\A^1$,
  and let $V_{d,k}\subset V_d\times V_k$ denote the locus of pairs $(h,g)$ such
  that \[(x^d+h)|(x^k+g).\] 
  Geometrically, $V_{d,k}$ is the space of 
  pairs $(z_d\leq z_k)$ of cycles or subschemes  on $\A^1$.
  This is obviously isomorphic to $V_d\times V_{k-d}$, via $x^k+g=(x^d+h)(x^{k-d}+h')$.
  In particular, $V_{d,k}$ and the projection $V_{d,k}\to V_d$ are smooth. 
 Consider the tangent space at the origin:
  \eqspl{t-d-k-eq}{
  T_{d,k}=T_0(V_{d,k})\subset V_d\oplus V_k.
  } This is a $k$-dimensional subspace.
 Though $V_{d,k}$ and $V_k$ are smooth and $k$-dimensional, 
 the projection
  $p:V_{d,k}\to V_k$ is ramified.
  The image of the differential at the origin $dp:T_{d,k}\to V_k$ can be identified
  as the space of polynomials divisible by $x^{\min(d, k-d)}$
  and therefore has dimension $\max(d, k-d)$. The projection $T_{d,k}\to V_d$ is surjective,
  which proves the Lemma.\end{proof}
For future reference, let  $T^0_{d,k}$ denote the kernel of the composite
surjection $T_{d,k}\to V_d\to V_d/(x^{d-1})$. This
$T^0_{d,k}$ is a subspace of codimension $d-1$ in $T_{d,k}$.
  It corresponds to (embedded) deformations of $(z_d\leq z_k)$ where $z_d$ deforms
  locally trivially.\par 
\par
The Lemma can be easily extended to the case of an arbitrary smooth ambient variety:
\begin{prop}\label{pairs-prop}Let $X$ be a smooth variety or complex manifold,
 and $X_0\sbr d,k.\subset X_0\sbr d.\times X_0\sbr k.$
be the locus of nested curvilinear schemes. The $X_0\sbr d,k.$ and the morphism $X_0\sbr d,k.\to X_0\sbr d.$
are smooth.

\end{prop}
\begin{proof}
We can work locally on $X$ at a punctual pair $z_d\leq z_k$ and choose local
  parameters $x_1,...,x_n$ so that
  \[\I_{z_d}=(x_1^d, x_2,...,x_n), \I_{z_k}=(x_1^k, x_2,...,x_n).\]
  Then deformations of these are given by
  \[x_1^d+h_1, x_2+h_2,...,x_n+h_n, h_i\in\sum\limits_{j=0}^{d-1}\C x_1^j, i=1,...,n,\]
  \[x_1^k+g_1, x_2+g_2,...,x_n+g_n, g_i\in\sum\limits_{j=0}^{k-1}\C x_1^j, i=1,...,n.\]
  The compatibility condition on local (arbitrary-order) deformations is
  \eqspl{}{
  x_1^d+h_1|x_1^k+g_1, h_i\equiv g_i\mod x_1^d, i=2,...,n.
  }
  The corresponding first-order conditions are
  \eqspl{}{
   (h_1, g_1)\in T_{d,k}, h_i\equiv g_i\mod x_1^d, i=2,...,n.
    }\par
  Canonically, we may identify these deformations as follows. Note the canonical map
  \eqspl{rho-eq}{
  \rho: N_{z_k/X}\oplus N_{z_d/X}\to N_{z_k/X}\otimes\O_{z_d}=\Hom(\I_{z_k/X}, \O_{z_d})
  }
 given on the first summand by restriction $\O_{z_k}\to\O_{z_d}$ and on the second summand by
 dualizing the inclusion $\I_{z_k,X}\subset \I_{z_d, X}$. Clearly $\rho$ is surjective.
 Then $H^0(\ker(\rho))$ is the canonical version of $T_{d,k}$.
 Furthermore, $\ker(\rho)$, identified with the mapping cone of $\rho$ and right-shifted by 1, 
 has the structure of semi-simplicial Lie algebra, equivalent to a Lie atom
 (see \cite{atom}, \cite{sela}), and the deformation theory of this structure
 is the deformation theory of the pair $(z_d, z_k)$. In particular,
 obstructions lie in $H^1(\ker(\rho))=0$.
   \end{proof}
  \section{The secant and contact sheaves and associated deformations}\label{secant-sheaf-sec}
\subsection{Non-excess case} \label{non-excess-subsec}
Our purpose here is to prove
\begin{prop}\label{non-excess-prop}
 Let $L, X$ be closed submaniflds of a smooth variety or complex manifold, $Y$ intersecting in a 
 finite-length scheme $Z$  with ideal $\I_Z=\I_X+\I_L$.
 Then there   subsheaves and
  Lie subatoms \[N^\ct\subset N^\rms\subset N_{L/Y}=\Hom(\I_L, \O_L)\]
 called the contact and secant sheaves, respectively, whose formation
 commutes with passage to an open subset,
 which control deformation of $L$ in $Y$ inducing locally trivial (resp. flat) deformations of $L\cap X$.
 The quotients $N_{L/Y}/N^\rms, N_{L/Y}/N^\ct$ are $\O_Z$-modules. 
 \end{prop}
\begin{proof} We define a the secant subsheaf 
 \[N^{\mathrm s}\subset N_{L/Y}=\Hom(\I_L. \O_L)\] 
 to consist
 of those homomorphisms that are compatible with $\I_X$, i.e. that map the subsheaf $\I_L\cap \I_X$
 to $\I_X.\O_L=\I_X/\I_X\cap\I_L$.
Thus the secant sheaf consists of (the triple of horizontal arrows in)
all commutative diagrams with exact columns
\eqspl{secant-eq}{
\begin{matrix}
0&&0\\
\downarrow&&\downarrow\\
\I_X\cap\I_L&\to&\I_X/\I_X\cap \I_L\\
\downarrow&&\downarrow\\
\I_L&\to&\O_L\\
\downarrow&&\downarrow\\
\I_L/\I_X\cap \I_L&\to&\O_Z\\
\downarrow&&\downarrow\\
0&&0
\end{matrix}
} 
Whereas $N_{L/Y}$ parametrizes local, first-order 
flat deformations of $L$ in $Y$, the secant subsheaf corresponds to those
deformations that induce a flat, i.e. length-preserving, deformation of the intersection scheme $Z$.
Thus $H^0(N^{\mathrm s})$ is the Zariski tangent space at $(L,Z)$ to the scheme $D^{\mathrm s}$ parametrizing such deformations,
while obstructions to such deformations are in $H^1(N^{\mathrm s}).$ Formally, $D^{\mathrm s}$ can
be defined as follows: let $D_1$ be  the germ at $[L]$  Hilbert scheme of $Y$(see \cite{sernesi}
\S 4.3) , 
and let $D_2$ be the germ at $[Z]$ of the Hilbert scheme of 
$X$. Then $D^{\mathrm s}\subset D_1\times D_2$ is the zero-scheme of the natural map
\[\I_1\to\O_2\]
where $\I_1$ is the universal ideal in $\O_Y$ and $\O_2$ is the universal length-$k$ quotient
pulled back from $X\sbr k.$.
A priori $D^{\mathrm s}$ parametrizes pairs $(Z'\subset L')$ such that
$Z'$ has the same length as $Z$ and $Z'\subset L'\cap X$ but by then semicontinuity, $Z'=L'\cap X$.\par
Assigning the triple
of horizontal arrows in\eqref{secant-eq} to the bottom arrow yields the natural map
\[\mathrm{res}_X:N^{\mathrm s}\to N_{Z/X}=\Hom(\I_{Z/X}, \O_Z)=\Hom(\I_L/\I_L\cap\I_X, \O_P/(\I_X+\I_L))\]
 which corresponds to the induced deformation of $Z$ as subscheme of $X$.
 We define the \emph{contact subsheaf} $N^{\mathrm c}\subset N^{\mathrm s}$ as the kernel of the composite map
 \eqspl{contact-eq}{N^{\mathrm s}\to N_{Z/X}\to T^1_Z.}
 The associated map $H^0(N^{\mathrm s})\to H^0(T^1_Z)$ is the derivative of the natural map of the deformation
 space $D^{\mathrm s}$ above to the abstract versal deformation of $Z$.
 Thus, $N^{\mathrm c}$ corresponds to
 the subscheme $D^{\mathrm c}\subset D^{\mathrm s}$ parametrizing
  deformations of $L$ which induce a \emph{locally trivial} deformation of $Z$.
 This will mainly be of interest when $Z$ has embedding dimension at most 1 or 2. Note that $N^{\mathrm s}$
 and $N^{\mathrm c}$,  in any event,  contain $\I_ZN_{L/P}$, hence $N/N^{\mathrm s}, N/N^{\mathrm c}$ are $\O_Z$-modules.\end{proof}


\subsection{Curvilinear non-excess case}\label{curvilin-subsec}
Our purpose here is to further analyze the secant and contact subsheaves
introduced above, in the case of curvilinear schemes. 
Set
\[ m=\dim(Y), c=\codim(X,Y), a=\dim(L), k=\l(Z), r=\#(\supp(Z)),\]
where $\l$ denotes length and $\#$ denotes cardinality.
Note $c\geq a$.
We will prove the following
\begin{prop}\label{curvilin-prop}
Notations as in Proposition \ref{non-excess-prop}, assume further that $Z$ is curvilinear.
Then the respective colengths of the secant and contact sheaves are given by
\eqspl{secant-contact-length-eq}{\l(N_{L/Y}/N^s)=k(c-a)\\
\l(N^s/N^c)=k-r.
}
\end{prop}

\begin{proof} We work locally at a point
where $Z$ has multiplicity $k_1>1$, as the case $k_1=1$ is similar and simpler. 
There, $Z$ is
contained in a smooth curve
$C\subset X$, which makes contact of order $k_1$ with $L$. We can write $C$ parametrically
as 
\[t\mapsto (y_1=t, y_2=t^{k_1},y_3=0,...,y_m= 0)\] 
where  $t$ is a coordinate on $C$,
$y_1,...,y_m$ are coordinates on $Y$, $y_1$ is a coordinate
on $L$ and the $y_2$ coordinate is conormal to $L$, i.e. $y_2=0$ on $L$.
Then we can find another set of local coordinates $x_1,..., x_m$ on $Y$ so that $L$ is defined
by $x_{a+1}=...=x_m=0$ and $\I_X$ is generated by
\[g_1=x_1^{k_1}+x_\eta, g_2=x_2, ...,g_c=x_c.\]
Note that $\I_X\cap\I_L$ is generated by $x_{a+1},...,x_c$.
Then $N_{L/Y}\otimes\O_Z$ is given locally by
\[ x_i\mapsto x_i+b_i, b_i\in\O_Z, i=a+1,...,m.\]
The subsheaf $N^{\mathrm s}\subset N_{L/Y}$ is
the unique subsheaf containing $\I_ZN_{L/Y}$ defined by the conditions, 
locally at every point of $\supp(Z)$, 
\[b_{a+1}=...=b_c=0\in\O_Z\]
while $N^{\mathrm c}\subset N^{\mathrm s}$ is 
the unique subsheaf containing $\I_ZN_{L/Y}$
defined by the additional condition in $\O_Z$:
\[b_N\equiv 0\mod x_1^{k_1-1}.\]
From this explicit description the dimension counts \eqref{secant-contact-length-eq} follow.
\end{proof}
\par 
\subsection{Curvilinear excess case}\label{excess}
Here notations are as in \S \ref{curvilin-subsec}, but we wish to consider motions of $L$ which keep just part
of the (curvilinear ) intersection $Z=X\cap L$. 
\begin{prop}\label{excess-prop}
Notations as in Proposition \ref{non-excess-prop}, assume
$L$ is a closed submanifold of a projective variety $P$ and that
$Y$ is an analytic open subset of $P$. Assume further that
 $Z=X\cap L$ is curvilinear, properly 
of type $(k.)$, let $(d.)\leq (k.)$ be a partition of weight $(d.)$, and let $W\subset Z$ be a
subscheme properly of type $(d.)$. Then\par
(i) there exists a sheaf $N^s_{L,W}$ supported on $L$, endowed with a Lie atom structure, and whose
formation commutes with localization, which controls flat deformations of the pair
$(L\subset Y, W\subset L\cap X)$;
\par (ii) there is a natural map $\gamma: N^s_{L,W}\to N_{L/Y}$ with finitely supported kernel
and cokernel, such that if $H^1(\im(\gamma))=0$ then the deformation space controlled by 
$N^s_{L, W}$ is smooth of its expected dimension, which is $h^0(N_{L/Y})-(c-a)d$;
\par
(iii) there is a subsheaf and subatom $N^{\ct}_{L, W} \subset N^\rms_{L,W}$ controlling deformations where
$W$ deforms locally trivially and if $H^1(\gamma(N^\ct_{L,W}))=0$ then the deformation space controlled by 
$N^\ct_{L,W}$ is smooth of its expected dimension, which is $h^0(N_{L/Y}-(c-a+1)d+\#(\supp(W))$.
\end{prop}
\begin{proof}
Let $D^{\mathrm s}_{(d.)}$ as in \S \ref{non-excess-subsec} be the space of deformations 
$(W'\subset L')$ of $(W\subset L)$ within $Y$ so that $W'\subset X$. 
Let $D^{\mathrm c}_{(d.)}\subset D^{\mathrm s}_{(d.)}$ be the subspace where $W$ deforms locally trivially, i.e.
$W'$ is (abstractly) isomorphic to $W$.
Locally at a point $p_1\in W$, deformations of $L$ can be described as above by
\[ x_i\mapsto x_i+b_i, b_i\in\O_L, i=a+1,...,m.\]
Note that local coordinates on $X$ are $x_1, x_{c+1},...,x_{m-1}$
Deformations of $W$ in $X$ are described by
sections of the normal sheaf $N_{W/X}$, a free $\O_W$-module,
or more concretely by
\[x_1^{k_1}\mapsto x_1^{k_1}+h_1, x_j\mapsto x_i+h_i, i=c+1,...,m-1, h_i\in\sum\limits_{j=0}^{d-1}\C x_j.\]
The compatibility of the two means (cf. \S \ref{pairs})
\eqspl{excess-secant-eq0}{
b_j\equiv 0\mod \I_W, j=a+1,...,c;\\
b_j\equiv h_j \mod \I_W, j=c+1,...,m-1;\\
x_1^{k_1}+b_m\equiv 0\mod \I_W+ (x_1^{d_1}+h_1)
}
Identifying $\O_W$ with $V_{d_1}$ as in \S\ref{pairs}, the last condition means exactly
that $(h_1, b_m)\in V_{d_1, k_1}$. The corresponding condition on first-order deformations is
that $(h_1, b_m)$ should be tangent to $V_{d_1, k_1}$ at 0, i.e. (cf. \S \ref{pairs}, \eqref{t-d-k-eq})
\[(h_1, b_m)\in T_0V_{d_1, k_1}=T_{d_1, k_1}\] where as we have seen, $T_{d_1, , k_1}$ is a codimension-$d_1$
subspace  of $ V_{d_1}\oplus V_{k_1}$. 
We may use the second condition in \eqref{excess-secant-eq0} to eliminate  $h_j, j=c+1,...,m-1$.
Note that $(h_1, h_{c+1}, ...,h+{m-1})$ may be identified with
and element of $H^0(N_{W/X})$.
Thus, the Zariski tangent space $T^{\mathrm s}$ to $D^{\mathrm s}_{(d.)}$ can be identified
with the subspace of $H^0(N_{L/Y})\oplus H^0(N_{W/X})$ defined by the conditions
\eqspl{excess-secant-tangential-eq}{
b_j\equiv 0\mod \I_W, j=a+1,...,c;\\
b_j\equiv h_j \mod \I_W, j=c+1,...,m-1;\\
(h_1, b_m)\in T_{d_1, k_1}.
}
Canonically, as in \S\ref{pairs}, especially \eqref{rho-eq}, letting
$i_W:W\to L$ denote the inclusion, we have a canonical surjection
of sheaves over $L$
\eqspl{excess-secant-eq}{
\rho_{L,W}:N_{L/Y}\oplus (i_W)_*N_{W/X}\to (i_W)_*i_W^*N_{L/Y}
}
We set
\eqspl{}{
N^{\mathrm s}_{L,W}=\ker(\rho_{L,W})
.}
In the language of Lie-theoretic deformation theory (\cite{atom}, \cite{sela}), 
the sheaf $N^{\mathrm s}_{L,W}$, or equivalently
the complex \eqref{excess-secant-eq} (right-shifted by 1), has the structure of semi-simplcial Lie
algebra ( \sela), in fact one equivalent to a Lie atom, and the deformation theory of
the pair $(L, W)$ where $L\subset Y, W\subset X\cap L$, is the deformation
theory associated to the \sela structure. Thus, first-order deformations
are in $H^0(N^{\mathrm s}_{L,W})$ while obstructions are in $H^1(N^{\mathrm s}_{L,W})$.\par
Then
\eqspl{}{
T^{\mathrm s}=H^0(N^{\mathrm s}_{L,W})
} 
Applying the snake lemma to the diagram
\eqspl{}{
\begin{matrix}
0\to & N^{\mathrm s}_{L,W}&\to & N_{L/Y}\oplus (i_W)_*N_{W/X}&\to & (i_W)_*i_W^*N_{L/Y}\to 0\\
&\downarrow&&\downarrow&&\\
& N_{L/Y}&=&N_{L/Y},&&&
\end{matrix}
}
we get an exact sequence
\eqspl{}{
0\to A\to N^{\mathrm s}_{L.W}\to N_{L/Y}\to B\to 0
}
where $A,B$ are respectively the kernel and cokernel of the map
\[N_{W/X}\to i_W^*N_{L/Y}.\] 
By an elementary computation from the explicit description above, we get
\eqsp{
\l_{p_i}(A)=\min(d_i, k_i-d_i), 
B\simeq A\oplus (c-a)\O_W.
}
Thus
\eqspl{}{\l(A)=\sum \min(d_i, k_i-d_i), \l(B)=(c-a)d+\sum \min(d_i, k_i-d_i).}
We will denote the image of $N^{\mathrm s}_{L,W}\to N_{L/Y}$ by $N^{\mathrm s}_0$. Its colength is
the length of $B$. Because $A$ has finite length, we have an exact sequence
\eqspl{}{
\exseq{H^0(A)}{T^{\mathrm s}}{H^0(N^{\mathrm s}_0)}.
}
This implies that if $N^{\mathrm s}_0$ is 'unobstructed' in the sense that
$H^i(N^{\mathrm s}_0)=0, \forall i>0$, then $H^i(N_{L/Y})=0$ too and
$T^{\mathrm s}$ has its expected dimension, viz. $h^0(N_{L/Y})-(c-a)d$.
\par
Next, note the natural map 
\[ \tau:N^{\mathrm s}_{L,W}\to (i_W)_*(T^1_W)\]
which by the above local computations is surjective.
Set
\eqspl{}{
N^{\ct}_{L,W}=\ker(\tau)
} As the mapping cone of $\tau $, this again has the structure of \sela \cite{sela}, and the inclusion
$N^{\mathrm ct}_{L,W}\to N^{\mathrm s}_{L,W}$ is a \sela homomorphism.
Then
\eqspl{}{
T^{\mathrm c}=H^0(N^{\ct}_{L,W}).
}
Thus we have an exact diagram
\eqspl{}{
\begin{matrix}
&0&&0&&&&\\
&\downarrow&&\downarrow&&&&\\
0\to&N^{\ct}_{L,W}&\to&N_{L/Y}\oplus (i_W)_*N'_{W/X}&\to& (i_W)_*i_W^*N_{L/Y}&\to 0\\
&\downarrow&&\downarrow&&||&\\
0\to&N^{\mathrm s}_{L,W}&\to&N_{L/Y}\oplus (i_W)_*N_{W/X}&\to& (i_W)_*i_W^*N_{L/Y}&\to 0\\
&\downarrow&&\downarrow&&&\\
&(i_W)_*T^1_W&=&(i_W)_*T^1_W&&&\\
&\downarrow&&\downarrow&&&&\\
&0&&0&&&&
\end{matrix}
}
where  $N'_{W/X}$ denotes the 'locally trivial deformations' subsheaf,
i.e. the kernel of $N_{W/X}\to T^1_W$ or equivalently, the image of
$T_X\otimes\O_W\to N_{W/X}$.
By an easy computation, we have an exact sequence
\eqspl{excess-sec-eq}{
\exseq{H^0(A^{\ct})}{T^{\mathrm c}}{H^0(N_0^{\ct})}}
where $A^{\ct}\subset A^{\mathrm s}$ has local colength  0 if $\l_{p_i}(A^{\mathrm c})<d_i$
and 1 otherwise, and the same for the colength of 
$N_0^{\ct}\subset N_0^{\mathrm s}$. Again it follows that if $H^i(N^{\mathrm c}_0)=0, \forall i>0$, then
$T^{\mathrm c}$ has its expected dimension, viz $h^0(N_{L/Y})-(c-a+1)d+\#(\supp W)$.\par 
\end{proof}
   \section{Secant rational curve Theorem}\label{proof}

We will use the notation and definitions of \S \ref{secants-and-fillers-sec}. 
    \begin{thm}\label{point-projection-thm}
   Let: \par 
   $Y$ be a smooth analytic open subset
of an irreducible $m$-dimensional  projective variety $P$, $X\subset Y$ a smooth closed $n$-dimensional subvariety, $y\in Y$ a general point;\par $\gg$ be an analytic open subset of the Hilbert scheme of curves in $P$ 
parametrizing nonsingular rational curves contained in $Y$, $\L\to\gg$ the
tautological family, $L$ a fibre of $\L$ containing $y$ and
    $(k.)$ a partition of weight $k$.\par Assume that
     $L$ is a proper $(k.)$-secant to $X$. 
      \par Then:\par
     
     (i) for any partition $(l.)$ refining $(k.)$ with exactly $r$ nonzero blocks, 
      $S_{(l.), Y}$ has the expected dimension, namely
        \eqspl{dimension}{m-3-K_Y.L-kc+r,  c:=m-n,}
     and its singularities at $L$
     are the same up to a smooth factor as those of of $D_{(l.)}$  at $L\cap X$;
     in particular, $S_{k, Y}$ and $S_{(k.),Y}$ are smooth at $L$;\par
     (ii) for any partition $(d.)\leq (k.)$ and subscheme $W\subset L\cap X$
     properly of type $(d.)$,   $\tilde S_{(d.), Y}$
     is smooth near $(L, W)$ and the normalization of  $S_{(d.), Y}$ is smooth at $L$;\par
     (iii) if $L$ is general among $(k.)$-secants to $X$, then each point in the support of $L\cap X$
     is general on some component of $X$.
   
%
%
%
%
    \end{thm}
    \begin{rem}
    The number \eqref{dimension} is the expected dimension of $S_{(l.), Y}$ because
    for a proper $(l.)$-secant,  it is the Euler
    characteristic of the relevant normal sheaf (atom), namely $N^\ct$,  controlling contact-preserving deformations,
    while non-proper $(l.)$-secants are 'expected' to be-and in fact are, by the following Corollary-
    specializations of proper ones.
    \end{rem}
    \begin{cor}
    For the general memeber $L$ of any component of $S_{(m.),Y}$, $L\cap X$ is properly of type $(m.)$.
    \end{cor}
    \begin{proof}[Proof of Corollary] Dimension count. More specifically, let
    $(k.)$ be the unique partition such that $L\cap X$ is properly
    of type $(k.)$. Then the partition $(k.)$ is obtained from $(m.)$ by a succession of\par
    - removing a wall, i.e. uniting two blocks, thus preserving total weight while reducing
    the number of blocks;\par
    - increasing the size of a block, thus increasing total weight and preserving cardinality
    of support.\par
    Hence if $(m.)\neq (k.)$, then $(\sum m_i)c-\# ({\mathrm{blocks\ of}} (m.))
    <(\sum k_i)c-\# ({\mathrm{blocks\ of}} (k.))$,
    so $L$ cannot be general.
    \end{proof}
    \begin{rem}
    A typical non-algebraic situation covered by the Theorem is where $Y$ is a tubular (analytic)
    neighborhood of a smooth complete rational curve in $P$, e.g. a projective space, and $X$ is a union of 
    some - not all- of the branches
    of an algebraic variety intersected with $Y$.
    \end{rem}
   
  \begin{proof}[Proof of Theorem] Let notations be as in the Theorem, and consider
  the secant and contact sheaves $N^\ct\subset N^\rms\subset N_{L/Y}$ as in
  Proposition \ref{non-excess-prop}. Because $y\in Y$ is general, the
  family of proper $(k.)$-secants to $X$ is filling for $Y$, and Lemma \ref{filling-lem}
  shows that $N^\ct$ is generated by global sections at $y$. Therefore $N^\ct$ is generically
  spanned, hence $N^\ct, N^\rms$ and $N_{L/Y}$ are all direct sums of line bundles
  of of nonnegative degree. This implies first that $S_{k, Y}$ is smooth at $L$ with 
  tangent space $H^0(N^\rms)$, and then that the local classifying morphism $\kappa$
  of $S_{k,Y}$ to 
  the abstract deformation space $\Def(z)$ is smooth at $L$. Because $S_{(l.),Y}=\kappa\inv(D_{(l.)})$,
  assertion (i) follows.\par
  Asserion (iii) follows from the fact that $N^\ct$ is globally generated.\par
  For (ii) we will similarly use Proposition \ref{excess-prop} (and its
  notations)  in lieu of \ref{non-excess-prop}. Our assumptions
imply that $\im(\gamma)$ is a direct sum of line bundles of nonnegative degree,
hence again $H^1(\im(\gamma))=0$ so we conclude as before the smoothness of $\tilde S_{(d,), Y}$
at $(L, W)$ for all possible subschemes $W$.  Because $W$ is properly 
of type $(d.)$, the fibre of the projection $\tilde S_{(d.), Y} \to S_{(d.), Y}$
through $(L, W)$ is the reduced point  $[W]$, the projection is unramified,
hence the normalization is smooth.
\end{proof}
\begin{rem}
As the referee points out, the argument above using Lemma \ref{filling-lem} replaces an inductive
argument used in a similar place by Gruson and Peskine in \cite{grp}.
\end{rem}
\begin{example}
As an obvious example, one can consider a $(2,d)$ complete intersection $X\subset \P^{n+2}, d\geq 3$.
Is has an oversize family of $d$-secants, namely the lines in the quadric, which
form a $(2n-1)$-dimensional family. Of course, this family is not filling for $\P^{n+2}$.
\end{example}
\begin{example}
Consider a plane $X=\P^2\subset\P^3$ and a transverse twisted cubic $L\subset\P^3$.
Viewing things in $\P^5$, let $Y$ be a generic quadric containing $X\cup L$. This is
easily checked to be smooth. The normal bundle $N_{L/Y}$ is of type $(4,3,3)$
and the secant bundle with respect to $X$ is of type $(3,2,2)$. The family of trisecant
twisted cubics to $X$ in $Y$ is unobstructed of dimension 14.
\end{example}
\begin{example}
For the (smooth) projected Veronese $X\subset\P^4$, there is a unique trisecant line $L$ through
a general point of $\P^4$ (see \cite{grp}, \S 5). Each point of $L\cap X$ is general on $X$.
\end{example}

 \section{Planar fibres}
\subsection{Statement}
  The main difficulty in extending Theorem \ref{point-projection-thm} to higher-dimensional
   secant flats and projections from (generic) higher-dimensional centers lies with the complexity
   of the finite schematic intersection $X\cap L$, which a priori is as ill-behaved as any finite scheme.
  This difficulty is still manageable when
 $X\cap L$ is locally planar, i.e. has embedding dimension 2 or less
 (e.g $L$ itself is 2-dimensional). This is due to Fogarty's theorem \cite{fogarty}
   about the smoothness of the Hilbert scheme of a smooth surface, and its consequence,
 Lemma \ref{length-lem} below, which shows that the secant sheaf $N^{\mathrm s}$ has
 the expected colength in $N_L$. Accordingly, we are able to prove
   Theorem \ref{line-projection-thm} below, which extends Theorem \ref{point-projection-thm}, essentially,
for ambient space $\P^N$ and locally planar fibres, such
as those which occur  upon projection from a generic line $\Lambda$ (i.e. have
the form $X\cap L$ where $L$ is a plane containing  $\Lambda$).
'Essentially' means we are able to control the locus of fibres of given length but not those of given
cycle type. Note, as a matter of terminology, that by
'locus of fibres' of a map we mean locus of point-images of fibres, a locus in
the \emph{target} of the map.\par For a subvariety $X\subset\pp^m$ and a linear $\lambda $-plane
$\Lambda$
disjoint from $X$, we denote by $X^\Lambda_k\subset\pp^{m-\lambda-1}$ the locus
of fibres of length $k$ or more of the projection
\[\pi_\Lambda:X\to\pp^{m-\lambda-1}.\]
Thus, $X^\Lambda_k$ is the locus of $(\lambda+1)$-planes 
containing $\Lambda$ and meeting $X$ in a scheme of length $k$ or more.
We denote by $\pi_\lambda, X^\lambda_k$ the analogous objects corresponding to the generic $\Lambda$.
 \begin{thm}\label{line-projection-thm}
 Let $X\subset \P^m$ be an irreducible closed subvariety of codimension $c>\lambda\geq 0$. Then  
 $X^\lambda_k$ is smooth of codimension
 $k(c-\lambda-1)$ in $\P^{m-\lambda-1}$, in a neighborhood of any
 point image of a  fibre of length exactly $k$ that is disjoint from
 the singular locus of $X$ and has embedding dimension 2 or less.
 \end{thm}
 \begin{rem}
 Note that the smoothness assertion of the Theorem applies to any fibre of length exactly $k$, including non-reduced ones.
 \begin{rem}
 The local planarity hypothesis is of course automatic when $\lambda\leq 1$ with the case $\lambda=0$ 
 being due already to Gruson-Peskine
 \cite{grp}. Alzati \cite{alzati} applied the Gruson-Peskine 
 theorem to $k$-normality of codimension-3 subvarieties of $\P^m$.
 Subsequently, Alzati (pers. comm.) was able to apply Theorem \ref{line-projection-thm},
 together with the method of \cite{alzati} to obtain a stronger result
 on $k$-normality, namely: a smooth, codimension-3 subvariety $X$ of degree $d$ and
 dimension 3 or more in $\P^m$ is $k$-normal provided \[k\geq d+1+(m-1)(m-2)(m-6)/3.\]
 \end{rem}
 \end{rem}
 \begin{cor}\label{planar-cor}
 Notations as in Theorem \ref{line-projection-thm}, assume
 \eqspl{ell}{
 \lambda<\min(c, c+2-n/3).
 } Then all fibres of $\pi_\lambda$ are planar, and
   for any $k\geq 1$, the locus of (point fibres of) fibres of $\pi_\lambda$ of length  $k$ is smooth of codimension
  $k(c-\lambda-1)$ in $\P^{m-\lambda-1}$, in a neighborhood of any fibre of length exactly $k$ that is disjoint from
  the singular locus of $X$.
 \end{cor}
 \begin{proof}[Proof of Corollary]  It suffices to prove
 the planarity assertion. If projection from $\Lambda$ has a fibre of embedding dimension
 3 or more at $x\in X$, then $\dim(\Lambda\cap T_xX)\geq 2$. An elementary dimension count shows that
 the locus of $\Lambda$ satisfying the latter incidence for some unspecified $x\in X$ is 
 of dimension at most
 \[4n-6+(\lambda-2)(n+c-\lambda),\] 
 which is less than $(\lambda+1)(n+c-\lambda)=\dim(\mathbb G(\lambda, n+c))$ provided
 $\lambda<c+2-n/3$.
 \begin{rem}\label{curvilin-proj-rem}
 Notations as above, a similar dimension count shows that whenever
 \eqspl{}{
 \lambda<\min(c, c+e-n/(e+1)),
 } any fibre of the generic projection of $X$ from a $\lambda$-plane either has embedding dimension
 at most $e$ or it meets
 the singular locus of $X$.
 \end{rem}
 
 \end{proof}
 \begin{cor}
 In the situation of Theorem \ref{line-projection-thm} or Corollary \ref{planar-cor},  if $X$ is smooth then\par
 (i)  for all $\lambda<\min(c, c+2-n/3)$, 
 the locus of fibres of $\pi_\lambda$ of 
 length exactly $k$ is smooth of codimension exactly  $k(c-\lambda-1)$ in $\P^{m-\lambda-1}$ or is empty; \par
 (ii)if $\lambda=1$, then the locus $X^1_k\subset\P^{m-2}$ of fibres of $\pi_1$ of length $k$ or more is of codimension $k(c-2)$
 or more, or is empty;\par
 (iii) if $\lambda=1$ and  $k(c-2)>m-2$, then $X^1_k$  is empty;
 \par (iv) if $n\leq 6$ and $\lambda<c$, then  $X^\lambda_k$  is  of codimension   $k(c-\lambda-1)$ 
 in $\P^{m-\lambda-1}$ 
 and smooth outside of $X^{\lambda}_{k+1}$.
 \end{cor}
 \begin{proof}
(i) Trivial from the foregoing Corollary.\par
(ii) Planarity is automatic and the fibres of greater length have greater codimension.
\par (iii) The codimension is too big.\par 
(iv) Immediate from \eqref{ell}.
 \end{proof}
\subsection{Proof of Theorem \ref{line-projection-thm}} 
Our point of view again is that a fibre of $\pi_\lambda$ comes from 
 a (line, plane) pair $(\Lambda, L)$ which is so mobile
that $\Lambda$ is a generic $\P^\lambda$ in $\P^m$ (such a pair may be said
to be 'infinitesimally filling with respect to $\Lambda$'). The proof of
Theorem \ref{line-projection-thm}  is a simple consequence of the two lemmas that follow. Each of them is stated 
 in somewhat greater generality than is required. The hard part is dealing with the local complexity of the finite
 scheme $L\cap X$. 
 \begin{lem}\label{length-lem} Let $X\subset P$ be a locally closed embedding of smooth varieties
 of respective dimensions $\nu-c, \nu$.
 Let $L$ be a smooth closed subvariety of dimension $\lambda+1$ in $P$ meeting $X$ in a scheme $Z$ of 
 finite length  $k$ and embedding dimension 2 or less, so that $X$ is closed in a neighborhood of $L$.
  Let $N^{\mathrm s}\subset N_{L/P}$ be the secant subsheaf, parametrizing deformations of $L$
  preserving the length of $X\cap L$ (cf. \S \ref{non-excess-subsec}). 
  Then the colength of $N^{\mathrm s}$ in $N_{L/P}$ is $k(c-\lambda-1)$.
 \end{lem}
 \begin{proof}
 Working locally at a point $z$ of $Z$, is will suffice to
 prove that the local colength of $N^{\mathrm s}$ is $\l_z(Z)(c-2)$. Now $N^{\mathrm s}$ is the kernel of the natural map
 \[\Hom(\I_L, \O_L)\to \Hom(\I_L\cap I_X, \O_Z)\]
 which factors through the surjection  $\Hom(\I_L,\O_L)\to \Hom(\I_L, \O_Z)$.
 The latter clearly has length $\l_z(Z)(\nu-\lambda-1), \nu=\dim(P)$.
 Now we have an exact sequence
 \[0\to\Hom(\I_L/(\I_L\cap I_X), \O_Z)\to \Hom(I_L, \O_Z)\to \Hom(\I_L\cap I_X, \O_Z) \]
 and
 \[\Hom(\I_L/(\I_L\cap I_X), \O_Z)=\Hom(\I_Z/\I_X, \O_Z) \]
Now the latter has length $\l_z(Z)\dim(X)$ due to the fact that $Z$ is a smooth point of the Hilbert scheme of $X$,
since it has embedding dimension 2 (an easy consequence of Fogarty's theorem on smoothness of the Hilbert scheme
of a smooth surface, see \cite{fogarty} or  \cite{lehn-montreal}, Cor. 3.4). Therefore locally the image of 
 \[\Hom(\I_L, \O_Z)\to \Hom(\I_L\cap I_X, \O_Z)\]
 hence also that of 
  \[\Hom(\I_L, \O_L)\to \Hom(\I_L\cap I_X, \O_Z)\]
  is of length $\l_z(Z)(\nu-\lambda-1-\dim(X))=\l_z(Z)(c-\lambda-1)$.
 \end{proof}
 In the following Lemma we consider a sheaf $M$ on $L$ (notations as above)
 such that $M\otimes\O_\Lambda$ admits a natural map from $N_{\Lambda/\P^N}$. 
 We will say that
 '$H^0(M)$ moves infinitesimally with $\Lambda$' if, given any $v_1\in H^0(M)$, there exists
 $v_2\in H^0(N_{\Lambda/\P^\nu})$ having the same image in $H^0(M\otimes\O_\Lambda)$.
 \begin{lem}\label{vanishing-lem} 
 Let $X\subset \P^\nu$ be a smooth codimension-$c$ locally closed subvariety,
 and let $\Lambda\subset\P^\nu$ be a generic linear $\P^\lambda$.
  Suppose  $L$ is an arbitrary linear $\P^{\lambda+1}$ in $\P^N$ containing $\Lambda$ and 
  meeting $X$ in a scheme $Z$ of 
  finite length  $k$.
   Let $M\subset N_{L/\P^\nu}$ be any coherent  subsheaf of $\O_L$-modules such that the support of $N_{L/\P^\nu}/M$
   is disjoint from $\Lambda$ and such that $H^0(M)$ deforms 
   infinitesimally with $\Lambda$.
 Then we have $H^1(M(t))=0$ for all $t\geq -1$.
 \end{lem}
 \begin{proof}
 By assumption,  $\Lambda$ moves as a generic $\P^\lambda$ of $\P^\nu$
 as $X$ is fixed,  (in particular $\Lambda\cap X=\emptyset$). 
 It follows first that $M\otimes\O_\Lambda=(\nu-\lambda-1)\O_\Lambda(1)$.
Our assumption that $H^0(M)$ deforms infinitesimally with $\Lambda$ means that 
  in the diagram
 \[\begin{matrix}
 &&H^0(N_{\Lambda/\P^\nu})\\
 &&\downarrow\\
 H^0(M)&\to&H^0(M\otimes\O_\Lambda)
 \end{matrix}
 \]
the image of the vertical map is contained in that of the horizontal map.
  Since the vertical map is obviously surjective, we conclude that the natural map
 \[H^0(M)\to H^0(M\otimes\O_\Lambda) \]
 is surjective. This implies that $H^1(M(-1))\to H^1(M)$ is injective 
 (in fact, an isomorphism, as $H^1(M\otimes\O_\Lambda)=0)$), and
 that $H^0(M(s))\to H^0(M(s)\otimes\O_\Lambda)$ is surjective for all $s\geq 0$,
 because the latter group is generated by $H^0(M\otimes\O_\Lambda)\otimes H^0(\O_\Lambda(s)))$.
 Consequently, $H^1(M(s-1))\to H^1(M(s))$ is injective for all $s\geq 0$, hence
 $H^1(M(s-1))=0$ for all $s\geq 0$. Actually, because $M\otimes\O_\Lambda=(\nu-\lambda-1)\O_\Lambda(1)$,
 it follows similarly  that $H^i(M(s-1))=0, \forall i>0, s\geq 0$.\end{proof}
 Now Theorem \ref{line-projection-thm} follows by putting together Lemma \ref{length-lem} for
 $L$ an $(\lambda+1)$- plane in $P=\P^m=\P^\nu$ and Lemma \ref{vanishing-lem} for  $M=N^{\mathrm s}$.
 Note $\Lambda$ being generic implies that the family $(L,\Lambda)$
 is filling, hence by Remark \ref{inf-filling-higher}, infinitesimally filling, with
 respect to $\Lambda$, which means precisely that
  $H^0(N^{\mathrm s})$ deforms infinitesimally with $\Lambda$, as
  in the hypothesis of Lemma \ref{vanishing-lem}.
 \par
 \begin{rem}
 It seems likely that a general locally planar secant $L\cap X$ of
 given length $k$ is actually reduced
 and perhaps even in uniform position or, if $X$ is non-degenerate, in general position
 within $L$. We have no proof of this. However  in the curvilinear case, $L\cap X$
 is reduced, as proved in the next section.
 \end{rem}
 \section{The curvilinear case}
In this section we will prove a (full, with contact conditions) generic projection
 theorem for centers of any dimension and fibres which are
\emph{curvilinear}. For $X,\Lambda$ as above, we denote by 
$X^\Lambda_{(k.)}$ the locus of (point images of) curvilinear fibres of type $(k.)$
disjoint from the singular locus of $X$, and let $X^ \lambda_{(k.)}$ denote 
$X^\Lambda_{(k.)}$ for $\Lambda$ generic of dimension $\lambda$.
  \begin{thm}\label{curvilinear-thm}
  Let $X\subset \P^m$ be an irreducible closed subvariety of codimension $c\geq 2$. Let
  \[\pi_\lambda:X\to\P^{m-\lambda-1}\] be the projection from a general $\lambda$-plane $\Lambda\subset\P^m$,
  $\lambda<c$.
  Then  for any $k_1+...+k_r=k\geq 1$, the  
  Thom-Boardman locus $X^\lambda_{(k.)}\subset \pp^{m-\lambda-1}$ of fibres of $\pi_\lambda$ of cycle type $(k_1,...,k_r)$ is smooth
  of codimension $k(c-\lambda)-r$ in $\P^{m-\lambda-1}$, in a neighborhood of any curvilinear fibre properly of cycle type
  $(k_1,...,k_r)$  that is disjoint from
   the singular locus of $X$.
    Moreover, for any partition $(l.)$ refining $(k.)$ , the singularities  of $X^\lambda_{(l.)}$ at $L$ are the same, up to a smooth factor,
    as those of the locus of divisors of type $(l.)$ on $\pp^1$ at a divisor properly of type $(k.)$, 
    and the normalization of $X^\lambda_{(l.)}$ is smooth.
    
  \end{thm}
  \begin{rem}
  In case $\lambda=0$, curvilinearity is automatic and Theorem \ref{curvilinear-thm} reduces to the
  Gruson-Peskine theorem. 
  \end{rem}
\begin{proof}[Proof of Theorem]
The proof is essentially identical to that of Theorem \ref{point-projection-thm}, using Lemma
\ref{vanishing-lem} to substitute for  the appropriate secant and contact sheaves being direct sums
of line bundles of nonegative degrees.
\end{proof}

 \begin{cor}\label{curvilinear-cor}
 Notations as in Theorem \ref{curvilinear-thm}, assume
 \eqspl{ell-curvilinear}{
 \lambda<\min(c, c+1-n/2).
 }
Then for any $k_1+...+k_r=k\geq 1$, 
 the locus of fibres of $\pi_\lambda$ of proper cycle type $(k_1,...,k_r)$ is smooth
   of codimension $k(c-\lambda)-r$ in $\P^{N-\lambda-1}$, in a neighborhood of any  fibre properly of cycle type
   $(k_1,...,k_r)$  that is disjoint from
    the singular locus of $X$; also, all these fibres are curvilinear.
 \end{cor}
 \begin{proof}
 Under assumption \eqref{ell-curvilinear}, Remark \ref{curvilin-proj-rem}
 shows that all fibres of $\pi_\lambda$  disjoint from the singular locus of $X$
 are curvilinear, so Theorem \ref{curvilinear-thm} applies..
 \end{proof}
 In particular, for $X$ of dimension up to 4,
 all generic projections that are morphisms, i.e. with $\lambda<c$, have only curvilinear fibres
 and good Thom-Boardman loci.
  \begin{example}
  If $X$ is a smooth 3-fold in $\P^7$, its generic projection to $\P^5 $ has a smooth double curve and no triple points.
  If $X$ is a smooth 3-fold in $\P^6$, its generic projection to $\P^4$ has a double surface, which is smooth outside
  of a triple curve, which itself is smooth outside of a finite number of 4-fold points. In fact, these points are ordinary, i.e.
  come from reduced 4-tuples on $X$. This  follows from Corollary \ref{curvilinear-cor}.
  \end{example}
  \begin{rem}
  It seems likely that the 'least common multiple' of Theorems \ref{curvilinear-thm} and \ref{point-projection-thm},
  as well as the corresponding generalization of Theorem \ref{line-projection-thm}
  hold (curvilinear or locally planar intersections $L\cap X$ in an arbitrary ambient space, 
  $L$ isomorphic to a projective space). The precise formulation and proof will have to await
  another occasion.
  \end{rem}
  \vskip 1cm
  \bibliographystyle{amsplain}
  \bibliography{mybib}
\end{document}

%% file: myheader2.tex
\def\inv{^{-1}}

\DeclareMathOperator{\codim}{codim}

\def\ct{\mathrm{ct}}
\def\rms{\mathrm{s}}

 \DeclareMathOperator{\Hom}{Hom}

\def\pp{\mathbb P}

\def\refp #1.{(\ref{#1})}

\newcommand{\A}{\mathcal{A}}

\def\sbr #1.{^{[#1]}}
\def\sfl #1.{^{\lfloor #1\rfloor}}

\newcommand\red{{\mathrm red}}

\def\inv{^{-1}}
\def\?{{\bf{??}}}

\def\dgla{differential graded Lie algebra\ }

\def\A{\Bbb A}

\def\C{\mathbb C}
\def\P{\mathbb P}

\def\sym{\text{\rm Sym} }

\def\Spec{\text{\rm Spec} }

\def\ss{\vskip.12in}

\def\L{\mathcal L}

\def\O{\mathcal O}

\def\g{\mathfrak g}

\def\1/2{\frac{1}{2}}

\def\I{\mathcal{ I}}

\def\im{\text{im}}

\def\2{{[2]}}
\def\l{\ell}

\def\<{\langle}
\def\>{\rangle}
\def\sela{SELA\ }

\def\im{\text{im}}
\def\2{{[2]}}
\def\l{\ell}

\def\scl #1.{^{\lceil#1\rceil}}
\def\spr #1.{^{(#1)}}
\def\sbc #1.{^{\{#1\}}}
\def\supp{\text{supp}}
\def\subpr#1.{_{(#1)}}

\renewcommand{\theequation}{\arabic{section}.\arabic{subsection}.\arabic{equation}}
\def\beq{\begin{equation*}}
\def\eeq{\end{equation*}}

\def\g3{{\Gamma\spr 3.}}
\def\gg{{\Gamma\spr 2.}}

\def\gg{\mathbb G}

\def\Def{\mathrm {Def}}

\newcommand{\eqspl}[2]{
\begin{equation}\label{#1}
\begin{split}
#2\end{split}\end{equation}}
\newcommand{\eqsp}[1]{\begin{equation*}
\begin{split}#1\end{split}\end{equation*}}

\newcommand{\exseq}[3]{
0\to #1\to #2\to #3\to 0
}
\newcommand{\beginalphaenum}{
\begin{enumerate}\renewcommand{\labelenumi}{ }
\item \begin{enumerate}
}

\def\eex{\end{rm}\end{example}}

\newtheorem{thm}{Theorem}  [section]

\newtheorem*{thm*}{Theorem}
\newtheorem*{prop*}{Proposition}
\newtheorem{cor}[thm]{Corollary}
\newtheorem*{cor*}{Corollary}

\newtheorem{lem}[thm]{Lemma}
\newtheorem{lem*}{Lemma}

\newtheorem*{claim*}{Claim}
\newtheorem{prop}[thm]{Proposition}

\newtheorem{defn}[thm]{Definition}
\theoremstyle{remark}

\newtheorem{rem}[thm]{Remark}
\newtheorem{crit-rem}[thm]{Critical remark}

\newtheorem{example}[thm]{Example}
\newtheorem*{example*}{Example}

\newtheorem*{defn*}{Definition}

%% file: filling3.bbl
\providecommand{\bysame}{\leavevmode\hbox to3em{\hrulefill}\thinspace}
\providecommand{\MR}{\relax\ifhmode\unskip\space\fi MR }
\providecommand{\MRhref}[2]{%
  \href{http://www.ams.org/mathscinet-getitem?mr=#1}{#2}
}
\providecommand{\href}[2]{#2}
\begin{thebibliography}{10}

\bibitem{alzati}
A.~Alzati, \emph{A new {C}astelnuovo bound for codimension three subvarieties},
  Arch. Math. \textbf{98} (2012), 219--227.

\bibitem{alzati-ottaviani}
A.~Alzati and G.~Ottaviani, \emph{The theorem of {M}ather on generic
  projections in the setting of algebraic geometry}, Man. math. \textbf{74}
  (1992), 391--412.

\bibitem{beis}
R.~Beheshti and D.~Eisenbud, \emph{Fibers of generic projections}, Comp. math
  \textbf{146} (2010), 435--456.

\bibitem{fogarty}
J.~Fogarty, \emph{Algebraic families on an algebraic surface}, Amer. J. Math.
  \textbf{90} (1968), 511--521.

\bibitem{grp}
L.~Gruson and C.~Peskine, \emph{On the smooth locus of aligned {H}ilbert
  schemes. {T}he $k$-secant lemma and the general projection theorem}, Duke
  math. J. (2013), \url{arxiv.org /1010.2399}.

\bibitem{lehn-montreal}
M.~Lehn, \emph{Lectures on {H}ilbert schemes}, CRM notes, Centre de
  {R}echerches {M}ath\'ematiques, Montreal, 2004.

\bibitem{lichtenbaum}
S.~Lichtenbaum and M.~Schlessinger, \emph{The cotangent complex of a morphism},
  Trans. AMS \textbf{128} (1967), 41--70.

\bibitem{mather}
J.~Mather, \emph{Generic projections}, Ann. Math. \textbf{98} (1973), 226--245.

\bibitem{structure}
Z.~Ran, \emph{Structure of the cycle map for {H}ilbert schemes of families of
  nodal curves}, 1--34, \url{http://arXiv.org/ 0903.3693}.

\bibitem{(n+2)sec}
\bysame, \emph{The (dimension+2)-secant lemma}, Invent. math \textbf{106}
  (1991), 65--71.

\bibitem{atom}
\bysame, \emph{Lie atoms and their deformation theory}, Geometric and
  Functional Analysis \textbf{18} (2008), 184--221.

\bibitem{sela}
\bysame, \emph{Jacobi-{B}ernoulli cohomology and deformations of schemes and
  maps}, C. Europ. J. Math. \textbf{10} (2012), 1541--1591.

\bibitem{sernesi}
E.~Sernesi, \emph{Deformations of algebraic schemes}, Grundl. d. math. Wiss.,
  vol. 334, Springer International, Berlin, Heidelberg, 2006.

\bibitem{zak}
F.L. Zak, \emph{Tangents and secants of algebraic varieties}, Transl. math.
  monog., vol. 127, Amer. math. soc., 1993.

\end{thebibliography}
